\documentclass[a4paper,11pt,twoside]{amsart}
\usepackage[english]{babel}
\usepackage[utf8]{inputenc}

\usepackage[a4paper,inner=3cm,outer=3cm,top=4cm,bottom=4cm]{geometry}
\usepackage{fancyhdr}
\usepackage{bold-extra}
\usepackage{mathrsfs}
\usepackage{comment}
\usepackage{graphics}
\usepackage{aliascnt}
\usepackage{hyperref}
\usepackage{amsmath}
\usepackage{amsfonts}
\usepackage{amssymb}
\usepackage{amsthm}
\usepackage{mathtools}
\usepackage{enumerate}
\usepackage{enumitem}

\newtheorem{theorem}{Theorem}[section]
\newtheorem{corollary}[theorem]{Corollary}
\newtheorem{lemma}[theorem]{Lemma}
\newtheorem{proposition}[theorem]{Proposition}
\theoremstyle{definition}

\numberwithin{equation}{section}

\newcommand\eps{\varepsilon}
\newcommand\R{\mathbb{R}}
\newcommand\Z{\mathbb{Z}}

\newcommand\C{\mathbb{C}}
\newcommand\T{\mathbb{T}}
\newcommand\wh{\widehat}
\newcommand\wt{\widetilde}
\newcommand\supp{\mathrm{supp}}
\newcommand\1{\mathbf{1}}

\hypersetup{hidelinks,pdftitle={Higher additive energies on discrete cubes},pdfauthor={Xuancheng Shao and Yu-Chen Sun}}

\begin{document}
\title{Higher additive energies on discrete cubes}
\author{Xuancheng Shao}
\address{Department of Mathematics, University of Kentucky\\
Lexington, KY 40506\\
USA}
\email{xuancheng.shao@uky.edu}
\thanks{XS was supported by NSF grant DMS-2452462}

\author{Yu-Chen Sun}
\address{
Department of Mathematics,
University of Bristol, Woodland Rd, Bristol BS8 1UG}
\email{yuchensun93@163.com}

\begin{abstract}
Let $m,n\geq2$ be integers. We study the least exponent $t_{m,n}$ such that the $m$-fold additive energy of any subset $A$ of the discrete cube $\{0,1,\cdots,n-1\}^d$ in any dimension $d$ satisfies $E_m(A)\leq |A|^{t_{m,n}}$. For every fixed $m$, we obtain the asymptotic formula
$$
 t_{m,n}=2m-1-(1+o(1))\log_n
 \frac{(2m-1)^{m-1/2}}{(2m)^{m-1}}
 \qquad(n\to\infty).
$$
For $m=2$, this gives
$$
t_{2,n} = 3 - (1+o(1)) \log_n\frac{3\sqrt{3}}{4},
$$
proving a conjecture of the first author.
\end{abstract}
\maketitle
%\tableofcontents

\section{Introduction}

For a nonempty finite set $A$ in an abelian group and an integer $m\geq2$, its
$m$-fold additive energy is defined by
$$
 E_m(A)=\big|\{(a_1,\ldots,a_{2m})\in A^{2m}:
 a_1+\cdots+a_m=a_{m+1}+\cdots+a_{2m}\}\big|.
$$
When $m=2$, this is the usual additive energy. The trivial bounds
$|A|^m\leq E_m(A)\leq |A|^{2m-1}$ hold for all $A$.
Understanding the structure of sets whose additive energy is large
relative to this upper bound is a central theme in additive
combinatorics. For $m=2$, the Balog--Szemer\'edi--Gowers theorem
produces a large subset with small doubling, and Freiman-type
theorems describe the structure of such subsets; see \cite{tao-vu}.
In this paper we study the improvement in the upper bound when
$A$ is contained in a discrete cube.

For integers $m,n\geq2$, let $t_{m,n}$ be the least real number such that
$$
 E_m(A)\leq |A|^{t_{m,n}}
$$
for every $d\geq1$ and every $A\subset\{0,1,\ldots,n-1\}^d$.
Kane and Tao \cite{kane-tao} proved that
$t_{2,2}=\log_2 6$. De Dios Pont, Greenfeld, Ivanisvili and
Madrid \cite{dios-pont-et-al} extended this identity to
$$
 t_{m,2}=\log_2\binom{2m}{m}\qquad(2\leq m\leq10),
$$
and Kova\v{c} \cite{kovac} subsequently proved the
same formula for every $m\geq2$. In these cases, the optimal exponent $t_{m,2}$ is attained when $A$ is the full binary cube $\{0,1\}^d$. For the ternary cube, however,
it was shown in \cite{dios-pont-et-al}
that $t_{2,3}>\log_3 19$, so the full cube is no longer extremal.
They also related the optimal exponent to a one-dimensional
Fourier inequality (see
\eqref{eq:HY-threshold} below).

The case $m=2$ and large $n$ was studied further by the first author in
\cite{shao-cubes}. It was proved there that
$$
 3-(1+o(1))\log_n\frac{3\sqrt3}{4}
 \leq t_{2,n}\leq 3-\log_n(1+c)
$$
for some absolute constant $c>0$, and it was conjectured that the
lower bound gives the correct asymptotic. Long
\cite[Corollary 1.2]{long} recently constructed explicit sets
attaining this lower bound, using lattice points in Euclidean
balls. We prove
the conjectured upper bound and obtain the analogous asymptotic
for every fixed $m\geq2$.

\begin{corollary}\label{cor:energy}
For every fixed integer $m\geq2$,
$$
 t_{m,n}=2m-1-(1+o(1))\log_n \frac{(2m-1)^{m-1/2}}{(2m)^{m-1}}
 \qquad(n\to\infty).
$$
In particular,
$$
 t_{2,n}=3-(1+o(1))\log_n\frac{3\sqrt3}{4}.
$$
\end{corollary}

The tensorization result of \cite{dios-pont-et-al} (see Proposition \ref{prop:tensorization} below) relates $t_{m,n}$ to the optimal exponent $q_m(n)$ in a one-dimensional Fourier inequality. For a finitely supported function $f:\Z\to\C$, define its Fourier transform $\wh{f}$ by
$$
 \wh f(\theta)=\sum_{x\in\Z}f(x)e(-x\theta),
 \qquad \theta\in\T=\R/\Z,
$$
where $e(y)=\exp(2\pi iy)$.

For $1\leq r<\infty$, we use the norms
$$
 \|f\|_{\ell^r(\Z)}=\left(\sum_{x\in\Z}|f(x)|^r\right)^{1/r},
 \qquad
 \|F\|_{L^r(X)}=\left(\int_X |F(x)|^r\,d\mu_X(x)\right)^{1/r},
$$
where $X\in\{\T,\R\}$ and $\mu_X$ is Lebesgue measure on $X$,
with $\T$ identified with $[0,1)$.
For $r=\infty$, the corresponding norms are
$$
 \|f\|_{\ell^\infty(\Z)}=\sup_{x\in\Z}|f(x)|,
 \qquad
 \|F\|_{L^\infty(X)}=\operatorname*{ess\,sup}_{x\in X}|F(x)|.
$$

Let $q_m(n)$ be the supremum of all $q\geq1$ for which
\begin{equation}\label{eq:HY-threshold}
 \|\wh f\|_{L^{2m}(\T)}\leq\|f\|_{\ell^q(\Z)}
\end{equation}
holds whenever $|\supp f|\leq n$. Set
\begin{equation}\label{eq:constants}
 p_m=\frac{2m}{2m-1},\qquad
 C_m=\frac{(2m-1)^{m-1/2}}{(2m)^{m-1}}.
\end{equation}
By the Hausdorff--Young inequality, \eqref{eq:HY-threshold} holds with $q = p_m$. Hence $q_m(n) \geq p_m$. The constant $C_m>1$ appears in the sharp Hausdorff--Young
inequality on $\R$. Babenko \cite{babenko} and Beckner
\cite{beckner} proved that, for every $F\in L^{p_m}(\R)$,
\begin{equation}\label{eq:euclidean-HY}
 \|\wh F\|_{L^{2m}(\R)}
 \leq C_m^{-1/(2m)}\|F\|_{L^{p_m}(\R)},
\end{equation}
with equality for Gaussian functions. Here the Fourier transform
on $\R$ is defined initially for $F\in L^1(\R)\cap L^{p_m}(\R)$ by
$$
 \wh F(\xi)=\int_\R F(x)e(-x\xi)\,dx,
 \qquad \xi\in\R,
$$
and extended by density to $L^{p_m}(\R)$.

\begin{theorem}\label{thm:main}
For every fixed integer $m\geq2$,
$$
 \lim_{n\rightarrow\infty} 2m\log n\Big(\frac1{p_m}-\frac1{q_m(n)}\Big) = \log C_m.
$$
More precisely, for every $0<\eps<1$, the following statements hold:
\begin{enumerate}[label=\textup{(\roman*)}]
\item For every $n$ sufficiently large in terms of $m,\eps$ and every function $f: \Z\rightarrow\C$ with $|\supp f| \leq n$, we have  $\|\wh{f}\|_{L^{2m}} \leq \|f\|_{\ell^q}$, where $q$ is defined by
$$
 \frac1q=\frac1{p_m}
 -\frac{(1-\eps)\log C_m}{2m\log n}.
$$
\item For every $n$ sufficiently large in terms of $m,\eps$, there exists a function $f:\Z\to[0,\infty)$ supported on exactly $n$ consecutive integers such that $\|\wh{f}\|_{L^{2m}} > \|f\|_{\ell^q}$, where $q$ is defined by 
$$
 \frac1q=\frac1{p_m}
 -\frac{(1+\eps)\log C_m}{2m\log n}.
$$
\end{enumerate}
\end{theorem}

Corollary \ref{cor:energy} follows immediately from Theorem \ref{thm:main} and the identity $t_{m,n}=2m/q_m(n)$ established in Proposition \ref{prop:tensorization} below. 

\subsection{Other additive properties of subsets of discrete cubes}

We give a brief summary of other related results in the literature.
Sumsets of subsets of discrete cubes have also been studied
extensively. Woodall \cite{woodall} and, independently,
Hajela and Seymour \cite{hajela-seymour} proved the sharp bound
$$
 |A+B|\geq(|A||B|)^{\log_2 3/2}
 \qquad(A,B\subset\{0,1\}^d).
$$
For general discrete cubes, Bourgain, Dilworth, Ford, Konyagin
and Kutzarova \cite{bourgain-et-al} proved
$|A+B|\geq(|A||B|)^{\tau_n}$ for
$A,B\subset\{0,\ldots,n-1\}^d$, where $\tau_n>0$ is determined by
$$
 n^{-2\tau_n}+\Big(\frac{n-1}{n}\Big)^{\tau_n}=1,
 \qquad 2\tau_n-1\sim\frac{\log2}{\log n}
 \quad(n\to\infty).
$$
More recently, Becker, Ivanisvili, Krachun and Madrid
\cite{becker-et-al} proved the sharp bound
$|A+B|\geq(|A||B|)^{\log5/(2\log3)}$ for
$A,B\subset\{0,1,2\}^d$. The binary and ternary bounds are attained
when both sets are the full cube.

There are also related results for Gowers norms on discrete
cubes. Beker, Crmari\'c and Kova\v{c} \cite{beker-crmaric-kovac}
study the optimal bounds for $\|f\|_{U^k}$ in terms of
$\|f\|_{\ell^q}$ when $f$ is supported on
$\{0,1,\ldots,n-1\}^d$. They prove the sharp bound
$$
 \|\1_A\|_{U^k}^{2^k}\leq |A|^{\log_2(2k+2)}
 \qquad(A\subset\{0,1\}^d),
$$
as well as two-sided bounds for the optimal exponents as
$n\to\infty$ with $k$ fixed, and a precise asymptotic as
$k\to\infty$ with $n$ fixed. The case $k=2$ is the additive
energy problem considered here with $m=2$.

\subsection{Proof outline and structure of the paper}

The main idea behind the proof of Theorem \ref{thm:main}(i) is
to distinguish between two types of functions $f$. If $f$ is
reasonably smooth, then we expect its piecewise constant
interpolation $F$ on $\R$ to satisfy
$$
 \|\wh f\|_{L^{2m}(\T)}\approx\|\wh F\|_{L^{2m}(\R)},
 \qquad \|f\|_{\ell^{p_m}}=\|F\|_{L^{p_m}}.
$$
The Babenko--Beckner inequality \eqref{eq:euclidean-HY} for $F$ gives a factor of
$C_m^{-1/(2m)}$. On the other hand, since $f$ is supported on
at most $n$ integers, H\"older's inequality gives
$$
 \|f\|_{\ell^{p_m}}
 \leq n^{1/p_m-1/q}\|f\|_{\ell^q}.
$$
Thus using the $\ell^q$ norm instead of the $\ell^{p_m}$ norm
costs at most a factor of $n^{1/p_m-1/q}$. These two factors
exactly cancel when
$$
 \frac1q=\frac1{p_m}-\frac{\log C_m}{2m\log n}.
$$
Section
\ref{sec:comparison} gives a precise comparison of the norms involving $f$ and $F$, with an error controlled by a discrete smoothness measure of $f$.

At the other extreme, when $f$ is close to a point mass, we
do not expect this comparison with a function on $\R$ to be
useful. Instead, we expect $\|f\|_{\ell^{p_m}} \approx \|f\|_{\ell^q}$ in this regime.
Lemma \ref{lem:local} proves the required inequality for functions $f$ near a
point mass, following \cite[Section 4.3]{shao-cubes}.

The rest of the paper is organized as follows.
Section \ref{sec:norms} records the norm inequalities used in the
proof, including an estimate for functions close to a point mass (Lemma \ref{lem:local}).
The comparison between discrete and continuous convolution is proved
in Section \ref{sec:comparison}. Section \ref{sec:compactness}
establishes the two limiting estimates used in the compactness argument in Section \ref{sec:main-proof}.
We complete the proof of Theorem \ref{thm:main} and Corollary
\ref{cor:energy} in Section \ref{sec:main-proof}.

\subsection*{AI assistance}

ChatGPT (OpenAI) was used by the authors to assist with developing mathematical arguments and preparing the manuscript. The authors take full responsibility for the results and their presentation.

\subsection*{Notation}
Throughout the paper, $m\geq2$ is fixed, and the constants
$p_m,C_m$ are given in \eqref{eq:constants}. We write $p=p_m$
and $s=2m-1=2m/p$. All logarithms without a specified base are natural. For a function $f:\Z\rightarrow\C$ and $h \in \Z$, define $\Delta_hf: \Z\rightarrow\C$ by the formula $\Delta_hf(x) = f(x+h) - f(x)$. 

For a finitely
supported nonnegative function $f$ on $\Z$, let $f^*$ be
obtained by placing the values of $f$ in nonincreasing order at
$$
 0,1,-1,2,-2,\ldots.
$$
In particular, we have
$$
f^*(0)\geq f^*(1)\geq f^*(-1)\geq f^*(2)\geq f^*(-2)\geq\cdots.
$$
We say that $f$ is \textbf{balanced} if $f = f^*$. For a nonnegative function of possibly infinite support, balanced means that its values satisfy the same ordering along $0,1,-1,2,-2,\ldots$.

\section{Hausdorff--Young and interpolation inequalities}\label{sec:norms}

In this section, we collect the norm inequalities used in the proof. As mentioned at the end of the introduction,
$m \geq 2$ and $p=2m/(2m-1)$.

\subsection{Hausdorff--Young and Young inequalities}

For $1\leq r\leq2$, write $r'$ for the conjugate exponent,
so that $1/r+1/r'=1$. The Hausdorff--Young inequality on $\Z$
states that
\begin{equation}\label{eq:HY-general}
 \|\wh f\|_{L^{r'}(\T)}\leq\|f\|_{\ell^r(\Z)}.
\end{equation}
This follows, for example,  by Riesz--Thorin interpolation between the trivial
$\ell^1\to L^\infty$ bound and Plancherel's theorem. On $\R$, we have the sharp form 
$$
 \|\wh F\|_{L^{r'}(\R)}
 \leq\frac{r^{1/(2r)}}{(r')^{1/(2r')}}\|F\|_{L^r(\R)}
 \qquad(1<r\leq2),
$$
proved by Babenko \cite{babenko} when $r'$ is an even integer and by Beckner \cite{beckner} for all $1<r\leq2$.

For functions $f,g:\Z\to\C$, define
$$
 (f*g)(x)=\sum_{y\in\Z}f(y)g(x-y),\qquad
 f^{*r}=\underbrace{f*\cdots*f}_{r\text{ times}}.
$$
We use the same notation for convolution on $\R$, with the sum
replaced by an integral. For example, for $F,G\in L^1(\R)$,
$$
 (F*G)(x)=\int_\R F(y)G(x-y)\,dy
 \qquad\text{for a.e. }x\in\R.
$$
Young's convolution inequality is
\begin{equation}\label{eq:young-general}
 \|f_1*\cdots*f_r\|_{\ell^u}
 \leq\prod_{j=1}^r\|f_j\|_{\ell^{p_j}},
 \qquad \frac1u=\sum_{j=1}^r\frac1{p_j}-(r-1),
\end{equation}
where $1\leq p_j,u\leq\infty$. The same inequality holds on
$\R$ with $L^u$ and $L^{p_j}$ norms. In particular,
$$
 \|f^{*m}\|_{\ell^2}\leq\|f\|_{\ell^p}^m.
$$
By Plancherel's theorem,
$$
 \|\wh f\|_{L^{2m}(\T)}^{2m}=\|f^{*m}\|_{\ell^2}^2.
$$
On $\R$, \eqref{eq:euclidean-HY} gives
\begin{equation}\label{eq:young-euclidean}
 \|F^{*m}\|_{L^2}^2\leq C_m^{-1}\|F\|_{L^p}^{2m}.
\end{equation}
This is also a special case of the sharp convolution inequality of
Beckner \cite{beckner} and Brascamp--Lieb \cite{brascamp-lieb}.

\subsection{Interpolation}

We shall repeatedly use the following consequence of H\"older's
inequality. On any measure space $X$ (such as $X = \Z$ with the counting measure or $X = \R$ with the Lebesgue measure), if
$1\leq a\leq r\leq b\leq\infty$ and
$1/r=\theta/a+(1-\theta)/b$ with $0\leq\theta\leq1$, then
\begin{equation}\label{eq:interpolation}
 \|f\|_{L^r(X)}
 \leq\|f\|_{L^a(X)}^\theta\|f\|_{L^b(X)}^{1-\theta}.
\end{equation}
In particular, taking $X=\Z$ with counting measure and one
exponent to be $\infty$ in \eqref{eq:interpolation} gives
\begin{equation}\label{eq:interpolation-infinity}
 \|f\|_{\ell^b}
 \leq\|f\|_{\ell^\infty}^{1-a/b}\|f\|_{\ell^a}^{a/b}
 \qquad(1\leq a\leq b\leq \infty).
\end{equation}
Since $\|f\|_{\ell^\infty}\leq\|f\|_{\ell^a}$, we have monotonicity
\begin{equation}\label{eq:interpolation-monotone}
\|f\|_{\ell^b}\leq\|f\|_{\ell^a}
\qquad(1\leq a\leq b\leq\infty).
\end{equation}
If $|\supp f|\leq n$, then H\"older's inequality gives the reverse bound
\begin{equation}\label{eq:support-norms}
 \|f\|_{\ell^a}\leq n^{1/a-1/b}\|f\|_{\ell^b}
 \qquad(1\leq a\leq b\leq\infty).
\end{equation}

\subsection{Functions close to a point mass}

Charalambides and Christ \cite[Theorem 1.3]{charalambides-christ}
proved that near equality in \eqref{eq:HY-general}, for any
fixed $1<r<2$, forces concentration at one point. For the
convolution inequality used here, this conclusion also
follows from the inverse Young inequality of Fournier
\cite{fournier}, a proof of which is given by Eisner and Tao
\cite[Proposition 5.4]{eisner-tao} (see also \cite[Proposition 4.5]{shao-cubes}).
By Plancherel's theorem, the following exact-equality case is an
immediate consequence of \cite[Theorem 1.3]{charalambides-christ}.

\begin{lemma}\label{lem:rigidity}
Let $f\in\ell^p(\Z)$. If
$$
 \|f^{*m}\|_{\ell^2}=\|f\|_{\ell^p}^m,
$$
then $f$ is supported on at most one point.
\end{lemma}

The preceding result shows that the Hausdorff--Young inequality
at exponent $p$ is strict unless $f$ is supported on at most
one point. For functions close to a point mass, the next lemma gives a stronger inequality with an exponent $q>p$.

\begin{lemma}\label{lem:local}
Fix $m\geq2$ and $\lambda>0$. Let $n$ be sufficiently large in terms of $m,\lambda$, and define $q$ by
$$
 \frac1q=\frac1p-\frac\lambda{2m\log n}.
$$
Let $f: \Z\rightarrow \C$ be a function with $|\supp f| \leq n$, and write $f=f(0)\1_{\{0\}}+h$. Suppose that
$$
 \|h\|_{\ell^p}\leq\delta |f(0)|
$$
for some $\delta > 0$ sufficiently small in terms of $m,\lambda$. Then
$$
 \|f^{*m}\|_{\ell^2}\leq\|f\|_{\ell^q}^m.
$$
\end{lemma}

\begin{proof}
The argument follows \cite[Section 4.3]{shao-cubes}.
If $f(0)=0$, then $h=0$ and the conclusion is trivial. We may therefore
assume that $f(0)\neq 0$. Set $x=|f(0)|$, $y=\|h\|_{\ell^p}$, and $t=y/x\leq\delta$. From the decomposition
$f=f(0)\1_{\{0\}}+h$ we obtain
$$
 f^{*m} = f(0)^m\1_{\{0\}}+mf(0)^{m-1}h
       +\sum_{r=2}^m\binom mr f(0)^{m-r}h^{*r}.
$$
The first two terms are orthogonal in $\ell^2$ because $h(0)=0$.
Since $\|h\|_{\ell^2}\leq\|h\|_{\ell^p}=y$, we have
$$
 \|f(0)^m\1_{\{0\}}+mf(0)^{m-1}h\|_{\ell^2}^2
 =x^{2m}+m^2x^{2m-2}\|h\|_{\ell^2}^2
 \leq x^{2m}(1+m^2t^2).
$$
For $1\leq r\leq m$, set $p_r=2r/(2r-1)\geq p$.
By Young's inequality \eqref{eq:young-general} and monotonicity \eqref{eq:interpolation-monotone}, we have
$$
 \|h^{*r}\|_{\ell^2}
 \leq\|h\|_{\ell^{p_r}}^r
 \leq\|h\|_{\ell^p}^r=y^r.
$$
Combining these estimates with the triangle inequality, we obtain
$$
\begin{aligned}
 \|f^{*m}\|_{\ell^2}
 &\leq  \|f(0)^m\1_{\{0\}}+mf(0)^{m-1}h\|_{\ell^2} + \sum_{r=2}^m\binom{m}{r} x^{m-r}\|h^{*r}\|_{\ell^2} \\
 &\leq x^m\Big(\sqrt{1+m^2t^2}
                      +\sum_{r=2}^m\binom mr t^r\Big).
\end{aligned}
$$
Since $\sqrt{1+m^2t^2}\leq1+m^2t^2/2$ and
$t^r\leq t^2$ for $r\geq2$ and $0\leq t\leq1$, we have
\begin{equation} \label{eq:local-upper}
 \|f^{*m}\|_{\ell^2} \leq x^m\Big(1 + \Big(\frac{m^2}{2} + 2^m\Big) t^2\Big).
\end{equation}
On the other hand, by H\"{o}lder's inequality (see \eqref{eq:support-norms}) we have
$$
 \|h\|_{\ell^q}\geq n^{1/q-1/p}\|h\|_{\ell^p}=e^{-\lambda/(2m)}y.
$$
Consequently,
\begin{equation}\label{eq:local-lower}
 \|f\|_{\ell^q}^m
 =\Big(x^q+\|h\|_{\ell^q}^q\Big)^{m/q} \geq x^m\Big(1+e^{-\lambda q/(2m)}t^q\Big)^{m/q}.
\end{equation}
We now compare the right-hand sides of \eqref{eq:local-upper} and \eqref{eq:local-lower}. Since $\delta$ is sufficiently small in terms of $m,\lambda$, we can ensure that
$$
 \Big(\frac{m^2}{2}+2^m\Big)\delta^{2-q}
 \leq e^{-\lambda q/(2m)}.
$$
Since $t \leq \delta$ and, by the definition of $q$, $m/q\geq1$
for sufficiently large $n$, it follows that
$$
1 + \Big(\frac{m^2}{2} + 2^m\Big) t^2 \leq 1 +  e^{-\lambda q/(2m)} t^q \leq \Big(1+e^{-\lambda q/(2m)}t^q\Big)^{m/q}.
$$
This completes the proof.
\end{proof}

\section{Comparison with convolution on \texorpdfstring{$\R$}{R}}\label{sec:comparison}

Given a finitely supported $v:\Z\to\C$, define its piecewise
constant interpolation $V:\R\to\C$ by
\begin{equation}\label{eq:step}
 V(x)=\sum_{k\in\Z}v(k)\chi(x-k),
\end{equation}
where $\chi = \1_{[0,1)}$.
Clearly $\|V\|_{L^r}=\|v\|_{\ell^r}$ for $1\leq r\leq\infty$. The purpose of this section is to compare norms of $V^{*m}$ and $v^{*m}$.
Define
$$
 \kappa_m(h)=\int_\R \chi^{*m}(x)\chi^{*m}(x-h)\,dx
 \quad(h\in\Z).
$$
These weights are nonnegative, symmetric in $h$, and supported
on $|h|<m$. We first compute their first and second moments.

\begin{lemma}\label{lem:kappa}
For $m\geq2$ we have
$$
 \sum_{h\in\Z}\kappa_m(h)=1,
 \qquad \sum_{h\in\Z}h^2\kappa_m(h)=\frac m{6}.
$$
\end{lemma}

\begin{proof}
For each integer $r\geq1$, set
$$
 K_r=\chi^{*r}*\wt\chi^{*r},\qquad
 \wt\chi(x)=\chi(-x).
$$
Then $\kappa_m(h)=K_m(h)$ and $K_r=K_1^{*r}$, where
$$
 K_1(x)=\max(1-|x|,0).
$$
Since $K_1$ is even and has integral $1$, the same is true for each $K_r$.
Write $\{x\} = x-\lfloor x\rfloor$. For any $x$, only $h=\lfloor x\rfloor$
and $h=\lfloor x\rfloor+1$ can contribute to the following sums, so
\begin{align*}
 & \sum_{h\in\Z}K_1(h-x)=(1-\{x\}) + \{x\} = 1,\\
 & \sum_{h\in\Z}h^2K_1(h-x)
 =(1-\{x\})\lfloor x\rfloor^2
     +\{x\}(\lfloor x\rfloor+1)^2=x^2+\{x\}(1-\{x\}).
\end{align*}
For $m \geq 2$ we have
$$
 \sum_{h\in\Z}K_m(h-x)
 =\int_{\R}K_{m-1}(y)
       \Big(\sum_{h\in\Z}K_1(h-x-y)\Big)\,dy=\int_{\R}K_{m-1}(y)\,dy=1.
$$
Setting $x=0$ proves the first identity in the lemma.

For the second moment, first note that $\int_\R x^2K_1(x)\,dx=1/6$ by direct computation. For $r\geq2$ we have
$$
\begin{aligned}
  \int_\R x^2K_r(x)\,dx
 &=\int_{\R^2}(y+z)^2K_{r-1}(y)K_1(z)\,dy\,dz \\
 &= \Big(\int_{\R}y^2K_{r-1}(y)\,dy\Big)
       \Big(\int_{\R}K_1(z)\,dz\Big)+\Big(\int_{\R}K_{r-1}(y)\,dy\Big)
       \Big(\int_{\R}z^2K_1(z)\,dz\Big)\\
 &\qquad + 2\Big(\int_{\R}yK_{r-1}(y)\,dy\Big)\Big(\int_{\R}zK_1(z)\,dz\Big) \\
 &=\int_\R y^2K_{r-1}(y)\,dy+\frac16,
 \end{aligned}
$$
where the mixed term vanishes since each $K_r$ is even.
It follows by induction that $\int_\R x^2K_r(x)\,dx=r/6$
for each $r$.

For $m\geq2$, we therefore obtain
\begin{align*}
 \sum_{h\in\Z}h^2\kappa_m(h)
 &=\int_\R K_{m-1}(x)
                   \sum_{h\in\Z}h^2K_1(h-x)\,dx\\
 &=\int_\R K_{m-1}(x) \Big(x^2 + \{x\}(1-\{x\})\Big)\,dx \displaybreak[1]\\
 &=\frac{m-1}{6}
       +\int_\R \{x\}(1-\{x\})K_{m-1}(x)\,dx\\
 &=\frac{m-1}{6}
       +\int_0^1u(1-u)\sum_{j\in\Z}K_{m-1}(j+u)\,du\\
 &=\frac{m-1}{6}+\int_0^1u(1-u)\,du=\frac m6.
\end{align*}
This completes the proof.
\end{proof}

Recall that for a function $f:\Z\rightarrow\C$ and $h \in \Z$, $\Delta_hf$ denotes the function defined by $\Delta_hf(x) = f(x+h)-f(x)$. The following proposition captures the idea that $v^{*m}$ is close to $V^{*m}$ when $v$ is reasonably ``smooth'' in the sense that $\Delta_1v^{*m}$ has small $\ell^2$ norm.

\begin{proposition}\label{prop:identity}
Let $v:\Z\to\C$ be finitely supported and let $V$ be defined by
\eqref{eq:step}.
Then
$$
 \|v^{*m}\|_{\ell^2}^2-\|V^{*m}\|_{L^2}^2
 =\frac12\sum_{h\in\Z}\kappa_m(h)
                    \|\Delta_h v^{*m}\|_{\ell^2}^2.
$$
In particular,
$$
 0\leq\|v^{*m}\|_{\ell^2}^2-\|V^{*m}\|_{L^2}^2
 \leq\frac m{12}\|\Delta_1 v^{*m}\|_{\ell^2}^2.
$$
\end{proposition}

\begin{proof}
Write $a=v^{*m}$. Expanding the convolution in \eqref{eq:step},
we obtain
$$
 V^{*m}(x)=\sum_{k\in\Z}a(k)\chi^{*m}(x-k).
$$
Hence
$$
 \|V^{*m}\|_{L^2}^2 = \langle V^{*m}, V^{*m}\rangle
 =\sum_{k,\ell\in\Z}a(k)\overline{a(\ell)}
      \int_\R\chi^{*m}(x-k)\chi^{*m}(x-\ell)\,dx
$$
By setting $h=\ell-k$ and recalling the definition of $\kappa_m(h)$, we obtain
$$
 \|V^{*m}\|_{L^2}^2 = \sum_{h\in\Z}\kappa_m(h)
                     \sum_{k\in\Z}a(k)\overline{a(k+h)}.
$$
For each $h$, let $T_ha$ be the function defined by $T_ha(k) = a(k+h)$, so that $\Delta_ha = T_ha-a$. Taking real parts of both sides of the equation above, we obtain
$$
 \|V^{*m}\|_{L^2}^2 = \sum_{h\in\Z}\kappa_m(h) \cdot \operatorname{Re}\langle a,T_ha\rangle.
$$
By translation invariance, we have $\|T_ha\|_{\ell^2} = \|a\|_{\ell^2}$  and hence
$$
\|\Delta_ha\|_{\ell^2}^2 = \langle T_ha-a,T_ha-a\rangle = 2\|a\|_{\ell^2}^2 - 2\operatorname{Re}\langle a,T_ha\rangle.
$$
It follows that
$$
 \|V^{*m}\|_{L^2}^2
 =\sum_{h\in\Z}\kappa_m(h)
       \Big(\|a\|_{\ell^2}^2-\frac12\|\Delta_ha\|_{\ell^2}^2\Big)
 =\|a\|_{\ell^2}^2
       -\frac12\sum_{h\in\Z}\kappa_m(h)\|\Delta_ha\|_{\ell^2}^2,
$$
where the second equality follows from
$\sum_h\kappa_m(h)=1$ in Lemma \ref{lem:kappa}. This proves the first assertion of the proposition.

For $h\geq1$, telescoping gives
$$
 \Delta_ha(k)=\sum_{j=0}^{h-1}\Delta_1a(k+j).
$$
The triangle inequality gives
$\|\Delta_ha\|_{\ell^2}\leq |h|\|\Delta_1a\|_{\ell^2}$ for
$h\geq1$. Since $\Delta_{-h}a(k)=-\Delta_ha(k-h)$, the same bound holds for all $h \in \Z$. Hence
$$
 0\leq\|v^{*m}\|_{\ell^2}^2-\|V^{*m}\|_{L^2}^2
  \leq\frac12\Big(\sum_{h\in\Z}h^2\kappa_m(h)\Big)
                         \|\Delta_1a\|_{\ell^2}^2
   =\frac m{12}\|\Delta_1a\|_{\ell^2}^2,
$$
where we used Lemma \ref{lem:kappa} in the last equality.
This completes the proof.
\end{proof}

We next use Proposition \ref{prop:identity} to prove a discrete
version of the sharp Euclidean inequality
\eqref{eq:young-euclidean}. For example, if $v$ is a balanced function with $v(0)=1$ and $\|v\|_{\ell^p}\rightarrow\infty$, then Proposition \ref{prop:comparison} leads to $\|v^{*m}\|_2^2 \leq (1+o(1)) C_m^{-1} \|v\|_{\ell^p}^{2m}$. Indeed, a nonnegative finitely supported balanced function satisfies $\|\Delta_1v\|_{\ell^1}=2v(0)=2\|v\|_{\ell^\infty}$, so the hypothesis of the proposition holds.

\begin{proposition}\label{prop:comparison}
Let $v:\Z\to[0,\infty)$ be finitely supported and suppose that
$$
 \|\Delta_1v\|_{\ell^1}\leq4\|v\|_{\ell^\infty}.
$$
Then
$$
 \|v^{*m}\|_{\ell^2}^2
 \leq C_m^{-1}\|v\|_{\ell^p}^{2m}
 +\frac{4m}{3}\|v\|_{\ell^\infty}^{2p}
                 \|v\|_{\ell^p}^{2m-2p}.
$$
\end{proposition}

\begin{proof}
Let $V$ be the interpolation in \eqref{eq:step}. Since
$\|V\|_{L^p}=\|v\|_{\ell^p}$, the sharp Euclidean inequality
\eqref{eq:young-euclidean} gives
$$
 \|V^{*m}\|_{L^2}^2\leq C_m^{-1}\|v\|_{\ell^p}^{2m}.
$$
By Proposition \ref{prop:identity}, it remains to estimate
$\|\Delta_1(v^{*m})\|_{\ell^2}$. Since
$\Delta_1(v^{*m})=(\Delta_1v)*v^{*(m-1)}$, an application of
Young's inequality yields
$$
 \|\Delta_1(v^{*m})\|_{\ell^2}
 \leq \|\Delta_1v\|_{\ell^1}\cdot\|v^{*(m-1)}\|_{\ell^2}.
$$
Set $r=2(m-1)/(2m-3)$. By Young's convolution inequality \eqref{eq:young-general} again, we have
$$
 \|v^{*(m-1)}\|_{\ell^2}
 \leq\|v\|_{\ell^r}^{m-1}.
 $$
Since $r>p$ and $(m-1)/r=(m-p)/p$, the interpolation
inequality \eqref{eq:interpolation-infinity} gives
$$
\|v\|_{\ell^r}^{m-1}
 \leq\|v\|_{\ell^\infty}^{p-1}\cdot\|v\|_{\ell^p}^{m-p}.
$$
Combining the three inequalities above, we obtain
\begin{align*}
 \|\Delta_1(v^{*m})\|_{\ell^2}
 &\leq\|\Delta_1v\|_{\ell^1}\cdot
           \|v\|_{\ell^\infty}^{p-1}\cdot\|v\|_{\ell^p}^{m-p}\\
 &\leq4\|v\|_{\ell^\infty}^{p}\cdot\|v\|_{\ell^p}^{m-p}.
\end{align*}
Substituting this bound into Proposition \ref{prop:identity} completes the proof.
\end{proof}

\section{A compactness argument}\label{sec:compactness}

In the proof of Theorem \ref{thm:main} in Section \ref{sec:main-proof}, we will consider a sequence $\{f_j\}$ of nonnegative balanced functions
with $\|f_j\|_{\ell^p}=1$, converging pointwise to $g$.
The function $g$ is also balanced, and Fatou's lemma gives
$\rho:=\|g\|_{\ell^p}^p\leq1$. The purpose of this section is to obtain estimates for certain norms of $f_j^{*m}$ in terms of $\rho$. We start with a lemma concerning the tail of $g$.

\begin{lemma}\label{lem:tail-product}
Let $g$ be a nonnegative balanced function in $\ell^p(\Z)$.
For an integer $R\geq1$, let $I_R = [-R,R] \cap \Z$ and define
$$
 S_R=\|g\1_{I_R}\|_{\ell^1},\qquad
 M_R=\sup_{x\notin I_R}g(x).
$$
Then $S_RM_R^{p-1}\to0$ as $R\to\infty$.
\end{lemma}

\begin{proof}
List the values of $g$ in decreasing order as
$a_1\geq a_2\geq\cdots\geq0$. Since $g$ is balanced, it suffices to prove that
$$
 \Big(\sum_{i\leq N}a_i\Big)a_{N+1}^{p-1}\longrightarrow0
$$
as $N\rightarrow\infty$.
Monotonicity  gives
$$
 \frac{N}{2}a_{N+1}^p\leq \sum_{i>\lfloor N/2\rfloor}a_i^p.
$$
Since $g \in \ell^p(\Z)$, the right-hand side above tends to $0$ as $N\rightarrow\infty$, and hence so does the left-hand side.
By H\"older's inequality,
\begin{align*}
 \Big(\sum_{i\leq N}a_i\Big)a_{N+1}^{p-1}
 &\leq N^{1-1/p}\Big(\sum_{i\leq N}a_i^p\Big)^{1/p}
                         a_{N+1}^{p-1}\\
 &\leq\|g\|_{\ell^p}(Na_{N+1}^p)^{1-1/p}\longrightarrow0.
\end{align*}
This completes the proof.
\end{proof}

We now obtain bounds for $\|f_j^{*m}\|_{\ell^2}$, depending on the behavior of the limit $g$. If $\rho \approx 0$, then $f_j$ is flat as $j$ grows and behaves like a continuous function, in which case we obtain savings by a factor of roughly $C_m^{-1}$. Recall that $s = 2m-1 = 2m/p$.

\begin{proposition}\label{prop:energy-limit}
Let $\{f_j\}$ be a sequence of finitely supported, nonnegative, balanced functions with
$\|f_j\|_{\ell^p}=1$. Suppose that $f_j(x)\to g(x)$ for every
$x\in\Z$, and set $\rho=\|g\|_{\ell^p}^p$. Then
$$
 \limsup_{j\to\infty}\|f_j^{*m}\|_{\ell^2}^2
 \leq\rho^s+C_m^{-1}(1-\rho)^s.
$$
\end{proposition}

\begin{proof}
For an integer $R\geq1$, let $I_R = [-R,R]\cap\Z$ and write
$$
 u_{j,R}=f_j\1_{I_R},\qquad
 v_{j,R}=f_j\1_{\Z\setminus I_R},\qquad
 g_R=g\1_{I_R}.
$$
Set
$$
 \rho_R=\|g_R\|_{\ell^p}^p,
 \qquad S_R=\|g_R\|_{\ell^1},
 \qquad M_R=\sup_{x\notin I_R}g(x).
$$
Thus $f_j = u_{j,R} + v_{j,R}$ and we have
$$
f_j^{*m} = u_{j,R}^{*m} + v_{j,R}^{*m} + \sum_{1 \leq b \leq m-1} \binom{m}{b}u_{j,R}^{*(m-b)}*v_{j,R}^{*b}.
$$
We first show that convolutions involving both $u_{j,R}$ and
$v_{j,R}$ make a negligible contribution in $\ell^2$-norm in the limit. Let $1\leq b\leq m-1$ and set $p_b=2b/(2b-1)>p$. By Young's convolution inequality \eqref{eq:young-general} and interpolation \eqref{eq:interpolation-infinity}, we get
\begin{align*}
 \|u_{j,R}^{*(m-b)}*v_{j,R}^{*b}\|_{\ell^2}
 &\leq\|u_{j,R}\|_{\ell^1}^{m-b}
                         \|v_{j,R}\|_{\ell^{p_b}}^b\\
 &\leq\|u_{j,R}\|_{\ell^1}^{m-b}
        \|v_{j,R}\|_{\ell^\infty}^{b-bp/p_b}
        \|v_{j,R}\|_{\ell^p}^{bp/p_b}\\
 &\leq \|u_{j,R}\|_{\ell^1}^{m-b}
        \|v_{j,R}\|_{\ell^\infty}^{b-bp/p_b}.
\end{align*}
As $j\to\infty$, we have $\|u_{j,R}\|_{\ell^1} \rightarrow S_R$ and $\|v_{j,R}\|_{\ell^{\infty}} \rightarrow M_R$ for fixed $R$, since the functions $f_j$ and $g$ are all balanced. More explicitly, $\|v_{j,R}\|_{\ell^\infty}=f_j(R+1)$ and $M_R=g(R+1)$. The identity $b-bp/p_b=(m-b)(p-1)$ therefore gives the required tail bound. By applying Lemma
\ref{lem:tail-product}, we obtain
$$
 \limsup_{j\to\infty}
 \|u_{j,R}^{*(m-b)}*v_{j,R}^{*b}\|_{\ell^2}
 \leq(S_RM_R^{p-1})^{m-b}=o_{R\to\infty}(1).
$$
It follows that
$$
 \limsup_{j\to\infty} \|f_j^{*m}\|_{\ell^2} \leq \limsup_{j\to\infty} \|u_{j,R}^{*m} + v_{j,R}^{*m}\|_{\ell^2} + o_{R\rightarrow\infty}(1).
$$
Since $\|u_{j,R}^{*m}\|_{\ell^2}$ and
$\|v_{j,R}^{*m}\|_{\ell^2}$ are at most $1$ by Young's
inequality, squaring both sides gives
$$
 \limsup_{j\to\infty}\|f_j^{*m}\|_{\ell^2}^2
 \leq\limsup_{j\to\infty}\|u_{j,R}^{*m}\|_{\ell^2}^2
       +\limsup_{j\to\infty}\|v_{j,R}^{*m}\|_{\ell^2}^2+2\limsup_{j\to\infty}
         \langle u_{j,R}^{*m},v_{j,R}^{*m}\rangle
       +o_{R\to\infty}(1).
$$
For the inner product, we have
$$
|\langle u_{j,R}^{*m},v_{j,R}^{*m}\rangle| \leq \|u_{j,R}^{*m}\|_{\ell^1}\|v_{j,R}^{*m}\|_{\ell^\infty} = \|u_{j,R}\|_{\ell^1}^m\cdot\|v_{j,R}^{*m}\|_{\ell^\infty}.
$$
Set $r=m/(m-1)>p$. By Young's inequality  \eqref{eq:young-general} and interpolation \eqref{eq:interpolation-infinity}, we have
$$
 \|v_{j,R}^{*m}\|_{\ell^\infty}
 \leq\|v_{j,R}\|_{\ell^r}^m\leq\|v_{j,R}\|_{\ell^\infty}^{m-mp/r}
        \|v_{j,R}\|_{\ell^p}^{mp/r}\leq\|v_{j,R}\|_{\ell^\infty}^{m-mp/r}.
$$
Note that $m-mp/r = m(p-1)$.
By applying Lemma \ref{lem:tail-product}, we obtain
$$
 \limsup_{j\to\infty}
 |\langle u_{j,R}^{*m},v_{j,R}^{*m}\rangle|
 \leq(S_RM_R^{p-1})^m=o_{R\to\infty}(1).
$$
For the $\ell^2$-norm of $u_{j,R}^{*m}$, note that
$u_{j,R}^{*m}\to g_R^{*m}$ pointwise and all these functions
are supported on the fixed interval $[-mR,mR]\cap\Z$. Hence
$$
\lim_{j\rightarrow\infty} \|u_{j,R}^{*m}\|_{\ell^2}^2 = \|g_R^{*m}\|_{\ell^2}^2 \leq \|g_R\|_{\ell^p}^{2m} = \rho_R^{s}.
$$
It follows that
\begin{equation}\label{eq:split-energy}
 \limsup_{j\to\infty}\|f_j^{*m}\|_{\ell^2}^2 \leq
 \rho_R^{s}
 +\limsup_{j\to\infty}\|v_{j,R}^{*m}\|_{\ell^2}^2
 +o_{R\to\infty}(1).
\end{equation}
For the tail $v_{j,R}$, since it is decreasing on $(R, +\infty)$ and increasing on $(-\infty,-R)$, we have $\|\Delta_1v_{j,R}\|_{\ell^1} \leq 4\|v_{j,R}\|_{\ell^{\infty}}$. Hence by Proposition \ref{prop:comparison} we obtain
$$
\|v_{j,R}^{*m}\|_{\ell^2}^2 \leq C_m^{-1}\|v_{j,R}\|_{\ell^p}^{2m} + \frac{4m}{3}\|v_{j,R}\|_{\ell^{\infty}}^{2p} \|v_{j,R}\|_{\ell^p}^{2m-2p}.
$$
Since $\|v_{j,R}\|_{\ell^p}^p = 1 - \|u_{j,R}\|_{\ell^p}^p \rightarrow 1-\rho_R$, we get
$$
 \limsup_{j\to\infty}\|v_{j,R}^{*m}\|_{\ell^2}^2
 \leq C_m^{-1}(1-\rho_R)^s
       +\frac{4m}{3}M_R^{2p}(1-\rho_R)^{s-2}.
$$
Since $M_R = o_{R\rightarrow\infty}(1)$, we conclude by substituting the above inequality into \eqref{eq:split-energy} that
$$
 \limsup_{j\to\infty}\|f_j^{*m}\|_{\ell^2}^2 \leq \rho_R^s + C_m^{-1}(1-\rho_R)^s + o_{R\rightarrow\infty}(1).
 $$
Finally, letting $R\rightarrow\infty$ completes the proof since $\rho_R\rightarrow\rho$.
\end{proof}

\begin{proposition}\label{prop:q-limit}
Under the hypotheses of Proposition \ref{prop:energy-limit},
suppose also that $|\supp f_j|\leq n_j$ and $n_j\to\infty$.
Fix $\lambda>0$ and, for all sufficiently large $j$, define $q_j$ by
$$
 \frac1{q_j}=\frac1p-\frac{\lambda}{2m\log n_j}.
$$
Then
$$
 \liminf_{j\to\infty}\|f_j\|_{\ell^{q_j}}^{2m}
 \geq\big(\rho+e^{-\lambda/s}(1-\rho)\big)^s.
$$
\end{proposition}

\begin{proof}
For an integer $R\geq1$, write $I_R = [-R,R]\cap\Z$ and
$\rho_R=\sum_{x\in I_R}g(x)^p$ (as in the proof of Proposition
\ref{prop:energy-limit}). Since $f_j(x)\to g(x)$ and $q_j\to p$, we have
$$
 \lim_{j\rightarrow\infty} \sum_{x\in I_R}f_j(x)^{q_j} = \rho_R.
$$
Now consider the contribution from $x \notin I_R$. Since there are at most $n_j$ nonzero terms, H\"{o}lder's inequality (see \eqref{eq:support-norms}) gives
$$
 \sum_{x\notin I_R} f_j(x)^{p}
 \leq n_j^{1-p/q_j}
       \Big(\sum_{x\notin I_R}f_j(x)^{q_j}\Big)^{p/q_j},
$$
which leads to
$$
 \sum_{x\notin I_R}f_j(x)^{q_j}
 \geq n_j^{1-q_j/p}
       \Big(\sum_{x\notin I_R}f_j(x)^p\Big)^{q_j/p}.
$$
Since $\sum_xf_j(x)^p=1$, we have
$$
 \lim_{j\rightarrow\infty} \sum_{x\notin I_R}f_j(x)^p
 = \lim_{j\rightarrow\infty} \Big(1-\sum_{x\in I_R}f_j(x)^p\Big)=1-\rho_R.
$$
The definition of $q_j$ gives
$$
 \frac{q_j}{p}=\Big(1-\frac{\lambda}{s\log n_j}\Big)^{-1},
 \qquad \Big(\frac{q_j}{p}-1\Big)\log n_j
 =\frac{\lambda/s}{1-\lambda/(s\log n_j)}
 \longrightarrow\frac\lambda s.
$$
Hence $\lim_{j\rightarrow\infty} n_j^{1-q_j/p} = e^{-\lambda/s}$. Combining the estimates
above and using $q_j\to p$, we obtain
$$
 \liminf_{j\to\infty}\sum_{x\notin I_R}f_j(x)^{q_j}
 \geq e^{-\lambda/s}(1-\rho_R).
$$
Adding the contribution from $I_R$ gives
$$
 \liminf_{j\to\infty}\sum_xf_j(x)^{q_j}
 \geq\rho_R+e^{-\lambda/s}(1-\rho_R).
$$
Raising both sides to the exponent $2m/q_j$ and using $2m/q_j\to s$ gives 
$$
 \liminf_{j\to\infty} \|f_j\|_{\ell^{q_j}}^{2m}
 \geq\Big(\rho_R+e^{-\lambda/s}(1-\rho_R)\Big)^{s}.
$$
Letting $R\to\infty$ completes the proof.
\end{proof}

\section{Proof of the main theorem}\label{sec:main-proof}

\subsection{Proof of Theorem \ref{thm:main}\textup{(i)}}\label{sec:upper}

We first recall the rearrangement inequality. For a finitely
supported nonnegative function $f$ on $\Z$, let $f^*$ be
obtained by placing the values of $f$ in nonincreasing order at
$$
 0,1,-1,2,-2,\ldots.
$$
Gabriel's rearrangement theorem
\cite[Theorem 4]{gabriel} gives
\begin{equation}\label{eq:rearrangement}
 \|\wh f\|_{L^{2m}(\T)}\leq
 \|\wh{f^*}\|_{L^{2m}(\T)}.
\end{equation}

Fix $0<\eps<1$ and set $\lambda=(1-\eps)\log C_m$. Suppose
that the conclusion fails. Then there exists a sequence $\{n_j\}$ with $n_j\rightarrow\infty$ and a sequence $\{f_j\}$ of functions $f_j: \Z\rightarrow \C$ with $|\supp f_j| \leq n_j$, such that
$$
 \|\wh{f_j}\|_{L^{2m}}^{2m} = \|f_j^{*m}\|_{\ell^2}^2>\|f_j\|_{\ell^{q_j}}^{2m},
 $$
where $q_j$ is defined by
$$
\frac{1}{q_j} = \frac{1}{p} - \frac{\lambda}{2m\log n_j}.
$$
By replacing $f_j$ with $|f_j|$, we may assume that $f_j$ takes nonnegative real values. This replacement cannot decrease the Fourier norm, because $|f_j^{*m}|\leq |f_j|^{*m}$ pointwise and Plancherel identifies the corresponding norms. By the rearrangement inequality, we may assume that $f_j = f_j^*$ is balanced. We may also normalize so that $\|f_j\|_{\ell^p}^p = 1$.

Since $0\leq f_j(x)\leq1$ for every $x\in\Z$, a diagonal subsequence converges pointwise to a function $g$. Then $g$ also takes nonnegative real values and is balanced.
By Fatou's lemma,
$\rho=\|g\|_{\ell^p}^p\in[0,1]$.

By Propositions \ref{prop:energy-limit} and \ref{prop:q-limit},
$$
 \limsup_j\|f_j^{*m}\|_{\ell^2}^2
 \leq\rho^s+C_m^{-1}(1-\rho)^s
$$
and
$$
 \liminf_j\|f_j\|_{\ell^{q_j}}^{2m}
 \geq(\rho+e^{-\lambda/s}(1-\rho))^s.
$$
Note that
$$
e^{-\lambda/s} = C_m^{-(1-\eps)/s} > C_m^{-1/s}.
$$
Hence if $\rho<1$ then
$$
 \rho^s+C_m^{-1}(1-\rho)^s
 \leq(\rho+C_m^{-1/s}(1-\rho))^s
 <(\rho+e^{-\lambda/s}(1-\rho))^s,
$$
contradicting the choice of $f_j$. Thus $\rho=1$.

Next we show that $f_j\to g$ in $\ell^p$. This follows
from the Br\'ezis--Lieb lemma \cite{brezis-lieb}, but in this
setting it can be proved directly as follows. For fixed $R$, pointwise
convergence on the finite set $I_R$ gives
$$
 \sum_{x\in I_R}|f_j(x)-g(x)|^p\longrightarrow0.
$$
On its complement, we have
\begin{align*}
 \limsup_{j\to\infty}\sum_{x\notin I_R}|f_j(x)-g(x)|^p
 &\leq2^{p-1}\lim_{j\to\infty}
       \sum_{x\notin I_R}\big(f_j(x)^p+g(x)^p\big)\\
 &=2^p\big(1-\|g\1_{I_R}\|_{\ell^p}^p\big).
\end{align*}
Since $\|g\|_{\ell^p}=1$, the last expression tends to zero as
$R\to\infty$. Combining the two estimates proves convergence
in $\ell^p$.

Next we show that $g = \1_{\{0\}}$.
Expanding the difference of the convolution powers and applying
Young's inequality \eqref{eq:young-general}, we obtain
\begin{align*}
 \|f_j^{*m}-g^{*m}\|_{\ell^2}
 &\leq\|f_j-g\|_{\ell^p}
       \sum_{r=0}^{m-1}\|f_j\|_{\ell^p}^r\|g\|_{\ell^p}^{m-1-r}\\
 &=m\|f_j-g\|_{\ell^p}\longrightarrow0.
\end{align*}
Moreover, Proposition \ref{prop:q-limit} with $\rho=1$ gives
$\liminf_j\|f_j\|_{\ell^{q_j}}\geq1$. Since
$\|f_j\|_{\ell^{q_j}}\leq\|f_j\|_{\ell^p}=1$, this forces
$\lim_j\|f_j\|_{\ell^{q_j}}=1$. By the construction of $f_j$,
we then have
$$
\liminf_{j\to\infty} \|f_j^{*m}\|_{\ell^2} \geq 1.
$$
Since $\|f_j^{*m}\|_{\ell^2} \leq \|f_j\|_{\ell^p}^m = 1$, this forces $\lim_j \|f_j^{*m}\|_{\ell^2}=1$, and hence $\|g^{*m}\|_{\ell^2}=1=\|g\|_{\ell^p}^m$.
By Lemma \ref{lem:rigidity}, $g$ is supported on one point.
Since $g$ is balanced,
we have $g=\1_{\{0\}}$.

Finally, write $f_j=x_j\1_{\{0\}}+h_j$, where $x_j=f_j(0)$.
Then $x_j\to1$ and $\|h_j\|_{\ell^p}\to0$. Lemma
\ref{lem:local} applies for sufficiently large $j$ and gives
$$
 \|f_j^{*m}\|_{\ell^2}^2\leq\|f_j\|_{\ell^{q_j}}^{2m},
$$
a contradiction.

\subsection{Proof of Theorem \ref{thm:main}\textup{(ii)}}\label{sec:lower}

We now prove Theorem \ref{thm:main}\textup{(ii)} by constructing a truncated discrete Gaussian. For $m=2$, this is precisely the construction used in
\cite[Proposition 2.2]{shao-cubes}. The same
argument applies for every fixed $m\geq2$. 

Let $n$ be sufficiently large in terms of $m,\eps$. Set $\lambda=(1+\eps)\log C_m$ and define $q=q_n$ by
$$
\frac{1}{q} = \frac{1}{p} - \frac{\lambda}{2m\log n}.
$$
 Take
$$
\begin{gathered}
 G(x)=e^{-\pi x^2},\qquad L_n=\frac{n}{3\sqrt{\log n}},\\[4pt]
 f_n(k)=G(k/L_n)
 \1_{\{-\lfloor(n-1)/2\rfloor\leq k\leq\lfloor n/2\rfloor\}}.
\end{gathered}
$$
Then $f_n$ is supported on exactly $n$ consecutive integers, for both odd and even $n$. Let $V_n$ be its piecewise constant interpolation as in
\eqref{eq:step}:
$$
V_n(x) = \sum_{k\in\Z} f_n(k) \1_{[0,1)}(x-k).
$$
Set $F_n(x)=V_n(L_nx)$. Then $F_n(x)\to G(x)$
for every $x\in\R$. 

By Proposition \ref{prop:identity}, we have
$$
 \|\wh f_n\|_{L^{2m}}
 =\|f_n^{*m}\|_{\ell^2}^{1/m}
 \geq\|V_n^{*m}\|_{L^2}^{1/m}
 =L_n^{1/p}\|\wh F_n\|_{L^{2m}}.
$$
Also, $\|f_n\|_{\ell^q}=\|V_n\|_{L^q}
=L_n^{1/q}\|F_n\|_{L^q}$. Consequently,
$$
 \frac{\|\wh f_n\|_{L^{2m}}}{\|f_n\|_{\ell^q}}
 \geq L_n^{1/p-1/q}
       \frac{\|\wh F_n\|_{L^{2m}}}{\|F_n\|_{L^q}}.
$$
Since $\log L_n/\log n\to1$, we have
$$
 L_n^{1/p-1/q}
 =\exp\Big(\frac{\lambda\log L_n}{2m\log n}\Big)
 \longrightarrow e^{\lambda/(2m)}.
$$

To justify the limiting norms, note that, for all sufficiently large $n$,
$L_n\geq1$ and $0\leq F_n(x)\leq1$. Whenever $F_n(x)\neq0$,
it equals $G(k/L_n)$ for an integer $k$ with
$|x-k/L_n|\leq L_n^{-1}\leq1$. Thus
$$
 F_n(x)\leq
 H(x):=\exp\big(-\pi(\max\{|x|-1,0\})^2\big).
$$
The function $H^p$ is integrable. Since $F_n(x)\to G(x)$,
dominated convergence gives $\|F_n-G\|_{L^p}\to0$.
The Hausdorff--Young inequality then implies
$$
 \|\wh F_n-\wh G\|_{L^{2m}}
 \leq C_m^{-1/(2m)}\|F_n-G\|_{L^p}\longrightarrow0.
$$
Moreover, $q=q_n\to p$ and $q_n\geq p$, so
$F_n(x)^{q_n}\leq H(x)^p$. Another application of dominated
convergence gives
$$
 \int_\R F_n(x)^{q_n}\,dx\longrightarrow\int_\R G(x)^p\,dx,
 \qquad \|F_n\|_{L^{q_n}}\longrightarrow\|G\|_{L^p}.
$$
Since $G$ attains equality in \eqref{eq:euclidean-HY}, we conclude that
$$
 \frac{\|\wh F_n\|_{L^{2m}}}{\|F_n\|_{L^q}}
 \longrightarrow\frac{\|\wh G\|_{L^{2m}}}{\|G\|_{L^p}}
 =C_m^{-1/(2m)}.
$$

Combining the two limits above, we get
$$
 \liminf_{n\to\infty}
 \frac{\|\wh f_n\|_{L^{2m}}}{\|f_n\|_{\ell^q}}
 \geq e^{\lambda/(2m)}C_m^{-1/(2m)}
 =C_m^{\eps/(2m)}>1.
$$

Consequently, the ratio exceeds $1$ for every sufficiently large $n$,
which proves part \textup{(ii)}.
Parts \textup{(i)} and \textup{(ii)}, together with monotonicity of
the $\ell^q$ norm in $q$, give
$$
 (1-\eps)\log C_m
 \leq 2m\log n\Big(\frac1p-\frac1{q_m(n)}\Big)
 \leq (1+\eps)\log C_m
$$
for all sufficiently large $n$. Letting $\eps\to0$ establishes
the limit asserted in Theorem \ref{thm:main}.

This completes the proof.

\subsection{Proof of Corollary \ref{cor:energy}}

Corollary \ref{cor:energy} follows immediately from Theorem \ref{thm:main} once we establish the following proposition, which is essentially \cite[Proposition 21]{dios-pont-et-al}.

\begin{proposition}\label{prop:tensorization}
For integers $m,n\geq2$, one has
$$
 t_{m,n}=\frac{2m}{q_m(n)}.
$$
\end{proposition}

\begin{proof}
\cite[Proposition 21]{dios-pont-et-al} shows  the relation
$$
 t_{m,n}=\frac{2m}{q_m'(n)},
$$
where $q_m'(n)$ is the supremum of all $q \geq 1$ for which \eqref{eq:HY-threshold} holds for all functions $f: \{0,1,\cdots,n-1\}\rightarrow\R$. Let $q = q_m'(n)$. We prove that \eqref{eq:HY-threshold} holds for all complex-valued functions $f$ with $|\supp f| \leq n$. Indeed, by replacing $f$ by $|f|$, we may assume that $f$ takes nonnegative real values. By the rearrangement inequality, we may assume that $f$ is supported on an interval of $n$ consecutive integers. By translation invariance, we may assume that $f$ is supported on $\{0,1,\cdots,n-1\}$. This proves that $q_m(n) = q_m'(n)$ as desired.
\end{proof}

\end{document}